\documentclass[a4paper,10pt,english]{smfart}
\usepackage[english]{babel}
\usepackage[OT2,T1]{fontenc}
\DeclareSymbolFont{cyrletters}{OT2}{wncyr}{m}{n}
\DeclareMathSymbol{\Sha}{\mathalpha}{cyrletters}{"58}
\usepackage{lmodern}
\usepackage{graphicx}
\usepackage{tabularx}
\usepackage{amsmath,amsthm,amssymb,amsfonts}
\usepackage[a4paper,left=4cm,right=4cm,top=4cm,bottom=4cm]{geometry}
\numberwithin{equation}{section}
\title[A statistical view on a conjecture of Lang]{A statistical view on the conjecture of Lang about the canonical height on elliptic curves}
\author{Pierre Le Boudec}
\subjclass{$11$D$45$, $11$G$05$, $11$G$50$, $14$G$05$}
\address{Departement Mathematik und Informatik \\ Fachbereich Mathematik \\ Spiegelgasse $1$ \\ $4051$ Basel \\ Switzerland}
\email{pierre.leboudec@unibas.ch}

\begin{document}

\makeatletter
\def\imod#1{\allowbreak\mkern10mu({\operator@font mod}\,\,#1)}
\makeatother

\newtheorem{lemma}{Lemma}
\newtheorem{theorem}{Theorem}
\newtheorem{corollary}{Corollary}
\newtheorem{proposition}{Proposition}
\newtheorem{conjecture}{Conjecture}
\newtheorem{conj}{Conjecture}
\renewcommand{\theconj}{\Alph{conj}}
\newtheorem{theo}{Theorem}
\renewcommand{\thetheo}{\Alph{theo}}

\newcommand{\vol}{\operatorname{vol}}
\newcommand{\D}{\mathrm{d}}
\newcommand{\rank}{\operatorname{rank}}
\newcommand{\Pic}{\operatorname{Pic}}
\newcommand{\Gal}{\operatorname{Gal}}
\newcommand{\meas}{\operatorname{meas}}
\newcommand{\Spec}{\operatorname{Spec}}
\newcommand{\eff}{\operatorname{eff}}
\newcommand{\rad}{\operatorname{rad}}
\newcommand{\sq}{\operatorname{sq}}
\newcommand{\tors}{\operatorname{tors}}
\newcommand{\Cl}{\operatorname{Cl}}
\newcommand{\Reg}{\operatorname{Reg}}
\newcommand{\Sel}{\operatorname{Sel}}
\newcommand{\corank}{\operatorname{corank}}
\newcommand{\an}{\operatorname{an}}

\begin{abstract}
We adopt a statistical point of view on the conjecture of Lang which predicts a lower bound for the canonical height of non-torsion rational points on elliptic curves defined over $\mathbb{Q}$. More specifically, we prove that among the family of all elliptic curves defined over
$\mathbb{Q}$ and having positive rank, there is a density one subfamily of curves which satisfy a strong form of Lang's conjecture.
\end{abstract}

\maketitle

\tableofcontents

\section{Introduction}

Adopting a statistical point of view on Diophantine problems often allows to gain insight into the questions at hand. This idea was recently illustrated in the context of the arithmetic of elliptic curves by the astonishing achievements of Bhargava and Shankar (see \cite{MR3272925}, \cite{MR3275847}, \cite{BhargavaShankar1} and \cite{BhargavaShankar2}) and of Bhargava, Skinner and Zhang (see \cite{MR3237733} and \cite{BhargavaSkinnerZhang}).

A famous conjecture of Lang predicts a lower bound for the canonical height of non-torsion rational points on elliptic curves defined over number fields. The goal of this article is to attack this conjecture over the field of rational numbers adopting the statistical point of view of Bhargava and his collaborators. We note here that we have restricted ourselves to the field $\mathbb{Q}$ as some of the tools involved in our work do not generalize immediately to the case of general number fields.

We start by recalling Lang's Conjecture \cite[Page $92$]{MR518817}. Let $E$ be an elliptic curve defined over $\mathbb{Q}$. The celebrated Mordell-Weil Theorem asserts that $E(\mathbb{Q})$ is a finitely generated abelian group. We respectively let
$\rank E(\mathbb{Q})$ and $E(\mathbb{Q})_{\tors}$ denote its rank and its torsion subgroup. In addition, we let $\hat{h}_E$ be the canonical height on $E$ (see Subsection~\ref{Subsection heights} for its definition) and finally we let $\Delta_{\min}(E)$ be the minimal discriminant of $E$.

\begin{conj}[Lang]
\label{Lang Conjecture}
There exists a constant $c > 0$ such that for any elliptic curve $E$ defined over $\mathbb{Q}$ and any
$P \in E(\mathbb{Q}) \smallsetminus E(\mathbb{Q})_{\tors}$, we have the lower bound
\begin{equation*}
\hat{h}_E(P) > c \log | \Delta_{\min}(E) |.
\end{equation*}
\end{conj}

Conjecture \ref{Lang Conjecture} is widely believed to be buried very deep. Nevertheless, Silverman has established the analog conjecture for elliptic curves with integral $j$-invariant (see \cite{MR630588}) and for families of twists of a fixed elliptic curve (see \cite{MR747871}). Moreover, the work of Hindry and Silverman \cite{MR948108} implies in particular that Conjecture \ref{Lang Conjecture} follows from Szpiro's Conjecture \cite{MR1065151}. It is also worth mentioning that Silverman \cite{MR2645981} has proved that a significantly weaker version of Szpiro's Conjecture also implies Conjecture~\ref{Lang Conjecture}. As a result, it seems unlikely that Conjecture~\ref{Lang Conjecture} implies the standard version of Szpiro's Conjecture.

We now introduce the material needed to state our results. For $(A,B) \in \mathbb{Z}^2$ such that $4 A^3 + 27 B^2 \neq 0$, we let $E_{A,B}$ be the elliptic curve defined over $\mathbb{Q}$ by the Weierstrass equation
\begin{equation*}
y^2 = x^3 + A x + B.
\end{equation*}
We also let $\mathcal{F}$ be the family of all elliptic curves $E_{A,B}$ where $(A,B) \in \mathbb{Z}^2$ satisfies the assumption that for any prime number $p$, we have either $p^4 \nmid A$ or $p^6 \nmid B$. In addition, we define the height of an elliptic curve
$E_{A,B} \in \mathcal{F}$ by
\begin{equation*}
H(E_{A,B}) = \max \{ 4 |A|^3, 27 B^2 \}.
\end{equation*}
Also, for $E_{A,B} \in \mathcal{F}$, we introduce the quantity $\mu(E_{A,B})$ defined by
\begin{equation*}
\mu(E_{A,B}) = \min \{ \hat{h}_{E_{A,B}}(P) : P \in E_{A,B}(\mathbb{Q}) \smallsetminus E_{A,B}(\mathbb{Q})_{\tors} \},
\end{equation*}
if $\rank E_{A,B}(\mathbb{Q}) \geq 1$ and $\mu(E_{A,B}) = \infty$ if $\rank E_{A,B}(\mathbb{Q}) = 0$. Finally, we set
\begin{equation*}
\mathcal{F}(X) = \{ E_{A,B} \in \mathcal{F} : H(E_{A,B}) \leq X \},
\end{equation*}
and
\begin{equation*}
\mathcal{F}_{\geq 1}(X) = \{ E_{A,B} \in \mathcal{F}(X) : \rank E_{A,B}(\mathbb{Q}) \geq 1\}.
\end{equation*}

We note that there exists a variant of Conjecture \ref{Lang Conjecture} where $\log | \Delta_{\min}(E) |$ is replaced by
$\max \{\log | \Delta_{\min}(E) |, h(j(E):1) \}$ where $h : \mathbb{P}^1(\mathbb{Q}) \to \mathbb{R}_{\geq 0}$ is the logarithmic Weil height and $j(E)$ is the $j$-invariant of the curve $E$. Silverman has pointed out to the author that this lower bound is often more useful for applications. Noting that we clearly have
\begin{equation*}
\max \{\log | \Delta_{\min}(E_{A,B}) |, h(j(E_{A,B}):1) \} \leq \log H(E_{A,B}) + \log 1728,
\end{equation*}
we see that our methods allow us to also deal with this stronger variant.

The main result of this article is the following.

\begin{theorem}
\label{Main Theorem}
Let $c < 7/24$ be fixed. We have the equality
\begin{equation*}
\lim_{X \to \infty}
\frac{\# \{ E_{A,B} \in \mathcal{F}_{\geq 1}(X) : \mu(E_{A,B}) > c \log H(E_{A,B}) \}}{\# \mathcal{F}_{\geq 1}(X)} = 1.
\end{equation*}
\end{theorem}

In other words, Theorem \ref{Main Theorem} states that among the family of all elliptic curves defined over $\mathbb{Q}$ and having positive rank, there is a density one subfamily of curves which satisfy a strong form of Conjecture \ref{Lang Conjecture}, where the word density is interpreted using the ordering of elliptic curves by height.

The author \cite{MR3455753} has recently studied the analogous problem for families of quadratic twists of a fixed elliptic curve. It is important to note that for these families, the analog of Conjecture \ref{Lang Conjecture} trivially holds (see for instance
\cite[Section $2$.$2$]{MR3455753}). This is the main reason why the problem investigated in the present article is substantially harder.

In addition, it should be stressed that there are many elliptic curves $E_{A,B}$ with rather small $\mu(E_{A,B})$. For instance, the analysis of the case of families of quadratic twists of a fixed elliptic curve given in \cite[Section $2$.$2$]{MR3455753} shows that there are infinitely many elliptic curves $E_{A,B}$ for which
\begin{equation*}
\mu(E_{A,B}) \leq \frac1{48} \log H(E_{A,B}) + O(1).
\end{equation*}
Theorem~\ref{Main Theorem} shows that families of elliptic curves satisfying such an inequality have to be sparse.

Let us now describe the main steps of the proof of Theorem \ref{Main Theorem}. According to the philosophy of Hindry and Silverman \cite{MR948108}, it is natural to try to get rid of the elliptic curves whose Szpiro ratios are not under control. Therefore, we start by showing that we can restrict our attention to a subfamily of elliptic curves which, loosely speaking, have the property that their discriminants are not far from being squarefree.

Then, the second step of the proof consists in checking that for the elliptic curves in this subfamily, the canonical height and the Weil height have comparable size. This is carried out by making use of the decomposition of the canonical height into a sum of N\'{e}ron local heights.

Putting these two first steps together, we can reduce the proof of Theorem \ref{Main Theorem} to the problem of counting the number of integral solutions to the Diophantine equation in five variables
\begin{equation*}
y^2 = w^3 + A w q^4 + B q^6,
\end{equation*}
where $w$, $y$, $q$, $A$ and $B$ are restricted to lie in boxes. This Diophantine equation is then viewed as two congruences modulo $w$ and $q^4$ respectively. Unfortunately, there seems to be no other option than dealing with the first one in a trivial way. Then, the second one is tackled using the work of Baier and Browning \cite{MR3100953}, which gives a nice upper bound for the number of solutions to an inhomogeneous cubic congruence and which is particularly efficient when the modulus of the congruence is divisible by large powers. We note that refining our strategy to deal with this counting problem may lead to an improvement of the constant $7/24$ appearing in
Theorem \ref{Main Theorem}. However, this does not seem to be an easy task.

Eventually, the last step of the proof consists in using the work of Bhargava and Skinner \cite{MR3237733} who established that, when elliptic curves defined over $\mathbb{Q}$ are ordered by height, a positive proportion have rank equal to $1$.

We would like to finish this introduction by providing evidence supporting the fact that a much stronger result than
Theorem~\ref{Main Theorem} should actually hold. To begin with, we note that the analogy with Goldfeld's Conjecture \cite{MR564926} about the average rank in families of quadratic twists of a fixed elliptic curve naturally leads to the expectation that
\begin{equation*}
\lim_{X \to \infty} \frac{\# \{ E_{A,B} \in \mathcal{F}_{\geq 1}(X) : \rank E_{A,B}(\mathbb{Q}) = 1\}}{\# \mathcal{F}_{\geq 1}(X)} = 1.
\end{equation*}
We remark that Goldfeld's Conjecture is supported by the Katz-Sarnak Philosophy (see \cite[Page $15$]{MR1659828}) about zeros of $L$-functions and also by Random Matrix Theory heuristics (see for instance \cite[Conjecture $1$]{MR1956231}). As a result, we can restrict our attention to elliptic curves $E_{A,B} \in \mathcal{F}_{\geq 1}(X)$ which have rank equal to $1$.

Let $\Sha(E_{A,B})$ and $\Omega(E_{A,B})$ respectively denote the Tate-Shafarevich group and the real period of an elliptic curve
$E_{A,B} \in \mathcal{F}_{\geq 1}(X)$. We recall that the Birch and Swinnerton-Dyer Conjecture predicts that the central value of the derivative of the Hasse-Weil $L$-function of a curve $E_{A,B} \in \mathcal{F}_{\geq 1}(X)$ with rank $1$ should essentially be equal to
\begin{equation*}
\Omega(E_{A,B}) \# \Sha(E_{A,B}) \mu(E_{A,B}).
\end{equation*}
Moreover, it is not hard to check that $\Omega(E_{A,B})$ has size about $H(E_{A,B})^{-1/12}$ (see for instance \cite[Section $1$]{MR717593}).

Then, we note that for families of quadratic twists of a fixed elliptic curve, an asymptotic formula for the average of the central values of the derivatives of the Hasse-Weil $L$-functions is known (see \cite[Theorem]{MR1081731} and \cite[Theorem $1$]{MR1109350}). Putting all these remarks together, we may conjecture that there exists $C > 0$ such that
\begin{equation*}
\frac1{\# \mathcal{F}_{\geq 1}(X)} \sum_{\substack{E_{A,B} \in \mathcal{F}_{\geq 1}(X) \\ \rank E_{A,B}(\mathbb{Q}) = 1}}
\# \Sha(E_{A,B}) \mu(E_{A,B}) \sim C X^{1/12} \log X.
\end{equation*}

Finally, the heuristics given by Delaunay \cite{MR1837670} and Bhargava, Kane, Lenstra, Poonen and Rains \cite{MR3393023} use different methods but both lead to the prediction that $\# \Sha(E_{A,B})$ should be small very often. We are thus led to make the following conjecture.

\begin{conjecture}
\label{Main Conjecture}
Let $\varepsilon > 0$ be fixed. We have the equality
\begin{equation*}
\lim_{X \to \infty}
\frac{\# \{ E_{A,B} \in \mathcal{F}_{\geq 1}(X) : \mu(E_{A,B}) > H(E_{A,B})^{1/12 - \varepsilon} \}}{\# \mathcal{F}_{\geq 1}(X)} = 1.
\end{equation*}
\end{conjecture}

We remark that Conjecture \ref{Main Conjecture} is the analog of \cite[Conjecture $1$]{MR3455753} which deals with families of quadratic twists of a fixed elliptic curve.

It is worth noting that it could be fair to attribute Conjecture~\ref{Main Conjecture} to Lang since this general philosophy was implicit in his investigation \cite[Section $1$]{MR717593}. More precisely, Conjecture~\ref{Main Conjecture} essentially states that the conjectural upper bound for $\mu(E_{A,B})$ stated in \cite[Conjecture~$3$]{MR717593} is almost optimal for most elliptic curves. However, following the referee's advice based on the fact that Conjecture~\ref{Main Conjecture} does not explicitly appear in the works of Lang, the author has not directly attributed it to him.

\subsection*{Acknowledgements}

It is a great pleasure for the author to thank Philippe Michel and Ramon Moreira Nunes for interesting conversations related to the topics of this article. In addition, the author is grateful to Fabien Pazuki and Joe Silverman for making useful comments on an earlier version of this article. Finally, the author would like to thank the referee for his or her careful work.

This work was initiated while the author was working as an Instructor at the \'{E}cole Polytechnique F\'{e}d\'{e}rale de Lausanne. The financial support and the wonderful working conditions that the author enjoyed during the four years he worked at this institution are gratefully acknowledged.

The research of the author is integrally funded by the Swiss National Science Foundation through the SNSF Professorship number $170565$ awarded to the project \textit{Height of rational points on algebraic varieties}. Both the financial support of the SNSF and the perfect working conditions provided by the University of Basel are gratefully acknowledged.

\section{Preliminaries}

\subsection{Restricting to a subfamily of the family of all elliptic curves}

We start by remarking that an elementary sieving argument (see for instance \cite[Lemma~$4$.$3$]{MR1176198}) shows that
\begin{equation}
\label{Size F}
\# \mathcal{F}(X) \sim \frac{2^{4/3}}{3^{3/2} \zeta(10)} X^{5/6},
\end{equation}
where $\zeta$ denotes the Riemann zeta function.

In addition, we define the discriminant of an elliptic curve $E_{A,B} \in \mathcal{F}$ by
\begin{equation*}
\Delta(E_{A,B}) = -16 (4 A^3 + 27 B^2).
\end{equation*}
For $n \in \mathbb{Z}_{\neq 0}$, there is a unique way to write $n = m \ell^2$ with $\ell \geq 1$ and $m$ squarefree. We denote by
$\sq(n)$ the only such integer $\ell$. For $\delta \in (0,1)$, we let
\begin{equation}
\label{Definition F_delta}
\mathcal{F}_{\delta}(X) = \left\{ E_{A,B} \in \mathcal{F}(X) :
\begin{array}{l}
|\Delta(E_{A,B})| > X^{1 - \delta} \\
\sq(\Delta(E_{A,B})) \leq 4 X^{\delta}
\end{array}
\right\}.
\end{equation}

The following lemma states that for any fixed $\delta \in (0, 1/6)$, the cardinality of the subset $\mathcal{F}_{\delta}(X)$ is asymptotically as large as the cardinality of the total set $\mathcal{F}(X)$.

\begin{lemma}
\label{Size F delta}
Let $\delta \in (0, 1/6)$ be fixed. We have the estimate
\begin{equation*}
\# \mathcal{F}_{\delta}(X) = \# \mathcal{F}(X) \left( 1 + O \left( \frac1{X^{5\delta/6}} \right) \right).
\end{equation*}
\end{lemma}

\begin{proof}
Recalling the asymptotic formula \eqref{Size F}, we see that our aim is to establish the upper bound
\begin{equation*}
\# \mathcal{F}(X) - \# \mathcal{F}_{\delta}(X) \ll X^{5/6 - 5\delta/6}.
\end{equation*}
In order to do so, we check that we have
\begin{equation}
\label{First}
\# \{ E_{A,B} \in \mathcal{F}(X) : |\Delta(E_{A,B})| \leq X^{1 - \delta} \} \ll X^{5/6 - 5\delta/6},
\end{equation}
and
\begin{equation}
\label{Second}
\# \{ E_{A,B} \in \mathcal{F}(X) : \sq(\Delta(E_{A,B})) > 4 X^{\delta} \} \ll X^{5/6 - 5\delta/6}.
\end{equation}

We start by proving the upper bound \eqref{First}. We have
\begin{equation*}
\# \{ E_{A,B} \in \mathcal{F}(X) : |\Delta(E_{A,B})| \leq X^{1 - \delta} \} \leq
\# \left\{(A,B) \in \mathbb{Z}^2 : \! \! \!
\begin{array}{l}
|A| \leq X^{1/3} \\
|4 A^3 + 27 B^2| \leq X^{1 - \delta}
\end{array}
\right\}.
\end{equation*}
First, it is clear that
\begin{equation}
\label{Upper bound 1}
\# \left\{(A,B) \in \mathbb{Z}^2 :
\begin{array}{l}
|A| \leq X^{1/3 - \delta/3} \\
|4 A^3 + 27 B^2| \leq X^{1 - \delta}
\end{array}
\right\} \ll X^{5/6 - 5\delta/6}.
\end{equation}
In addition, if $|A| > X^{1/3 - \delta/3}$ then the estimate
\begin{equation*}
B^2 = - \frac{4}{27} A^3 \left( 1 + O \left( \frac{X^{1 - \delta}}{|A|^3} \right) \right)
\end{equation*}
implies that
\begin{equation*}
|B| = \frac{2}{27^{1/2}} |A|^{3/2} + O \left( \frac{X^{1 - \delta}}{|A|^{3/2}} \right).
\end{equation*}
Therefore, we get
\begin{equation*}
\# \left\{(A,B) \in \mathbb{Z}^2 :
\begin{array}{l}
X^{1/3 - \delta/3} < |A| \leq X^{1/3} \\
|4 A^3 + 27 B^2| \leq X^{1 - \delta}
\end{array}
\right\} \ll \sum_{X^{1/3 - \delta/3} < |A| \leq X^{1/3}} \left( \frac{X^{1 - \delta}}{|A|^{3/2}} + 1 \right),
\end{equation*}
which eventually gives
\begin{equation}
\label{Upper bound 2}
\# \left\{(A,B) \in \mathbb{Z}^2 :
\begin{array}{l}
X^{1/3 - \delta/3} < |A| \leq X^{1/3} \\
|4 A^3 + 27 B^2| \leq X^{1 - \delta}
\end{array}
\right\} \ll X^{5/6 - 5\delta/6}.
\end{equation}
The upper bound \eqref{First} directly follows from the upper bounds \eqref{Upper bound 1} and \eqref{Upper bound 2}.

We now prove the upper bound \eqref{Second}. First, we note that
\begin{equation}
\label{zero}
\# \{ E_{A,B} \in \mathcal{F}(X) : AB=0 \} \ll X^{1/2}.
\end{equation}
We introduce the set
\begin{equation*}
\mathcal{E}_{\delta}(X) = \left\{(A,B,m,\ell) \in \mathbb{Z}_{\neq 0}^3 \times \mathbb{Z}_{\geq 1} :
\begin{array}{l}
|A|^3, B^2 \leq X  \\
\ell > X^{\delta} \\
4 A^3 + 27 B^2 = m \ell^2
\end{array}
\right\}.
\end{equation*}
The upper bound \eqref{zero} shows that
\begin{equation*}
\# \{ E_{A,B} \in \mathcal{F}(X) : \sq(\Delta(E_{A,B})) > 4 X^{\delta} \} \ll \# \mathcal{E}_{\delta}(X) + X^{1/2}.
\end{equation*}
As a result, in order to establish the upper bound \eqref{Second}, it suffices to check that we have
\begin{equation}
\label{Second'}
\# \mathcal{E}_{\delta}(X) \ll X^{5/6 - 5\delta/6}.
\end{equation}
We define
\begin{equation*}
\mathcal{E}_{\delta}^{(1)}(X) = \{(A,B,m,\ell) \in \mathcal{E}_{\delta}(X) : \ell > X^{1/3} \},
\end{equation*}
and $\mathcal{E}_{\delta}^{(2)}(X) = \mathcal{E}_{\delta}(X) \smallsetminus \mathcal{E}_{\delta}^{(1)}(X)$, so that
\begin{equation}
\label{E}
\# \mathcal{E}_{\delta}(X) = \# \mathcal{E}_{\delta}^{(1)}(X) + \# \mathcal{E}_{\delta}^{(2)}(X).
\end{equation}
We start by proving an upper bound for the cardinality of the set $\mathcal{E}_{\delta}^{(1)}(X)$. We note that we have $m \ell^2 \ll X$. Therefore, the condition $\ell > X^{1/3}$ implies that $m \ll X^{1/3}$. Since $A m \neq 0$ and $A, m \ll X^{1/3}$, it follows from the work of Estermann \cite{MR1512732} that we have
\begin{equation*}
\# \left\{(B,\ell) \in \mathbb{Z}_{\neq 0} \times \mathbb{Z}_{\geq 1} :
\begin{array}{l}
B, \ell \ll X^{1/2}  \\
4 A^3 + 27 B^2 = m \ell^2
\end{array}
\right\} \ll X^{\varepsilon},
\end{equation*}
for any fixed $\varepsilon > 0$, which finally gives
\begin{equation}
\label{E1}
\# \mathcal{E}_{\delta}^{(1)}(X) \ll X^{2/3 + \varepsilon}.
\end{equation}
We now prove an upper bound for the cardinality of the set $\mathcal{E}_{\delta}^{(2)}(X)$. To do so, we write
\begin{equation*}
\# \mathcal{E}_{\delta}^{(2)}(X) \leq \# \left\{(A,B,\ell) \in \mathbb{Z}_{\neq 0}^2 \times \mathbb{Z}_{\geq 1} :
\begin{array}{l}
|A|^3, B^2 \leq X \\
X^{\delta} < \ell \leq X^{1/3} \\
27 B^2 = -4 A^3 \imod{\ell^2}
\end{array}
\right\}.
\end{equation*}
Setting $d=\gcd(B,\ell)$, we see that $d^2 \mid 4 A^3$ and $\gcd(4A^3/d^2, \ell^2/d^2) \mid 27$. We thus have
\begin{equation*}
\# \mathcal{E}_{\delta}^{(2)}(X) \ll \sum_{1 \leq |A| \leq X^{1/3}} \ \sum_{d^2 \mid 4 A^3}
\sum_{\substack{X^{\delta} < d n \leq X^{1/3} \\ \gcd(4A^3/d^2, n^2) \mid 27}}
\# \left\{b \in \mathbb{Z}_{\neq 0} :
\begin{array}{l}
|b| \leq X^{1/2}/d \\
27 b^2 = \alpha \imod{n^2}
\end{array}
\right\},
\end{equation*}
where we have set $\alpha = -4A^3/d^2$. Since $\gcd(\alpha,n^2)$ is bounded by an absolute constant, we can make use of the classical upper bound
\begin{equation*}
\# \left\{b \in \mathbb{Z}_{\neq 0} :
\begin{array}{l}
|b| \leq X^{1/2}/ d \\
27 b^2 = \alpha \imod{n^2}
\end{array}
\right\} \ll n^{\varepsilon} \left( \frac{X^{1/2}}{d n^2} + 1 \right),
\end{equation*}
for any fixed $\varepsilon > 0$. We thus obtain
\begin{equation*}
\# \mathcal{E}_{\delta}^{(2)}(X) \ll X^{\varepsilon} \sum_{1 \leq |A| \leq X^{1/3}} \ \sum_{d^2 \mid 4A^3} 
\left( X^{1/2-\delta} + X^{1/3} \right).
\end{equation*}
Using the divisor bound, we get
\begin{equation}
\label{E2}
\# \mathcal{E}_{\delta}^{(2)}(X) \ll  X^{5/6 - \delta + \varepsilon} + X^{2/3+\varepsilon}.
\end{equation}
As a result, putting together the equality \eqref{E} and the upper bounds \eqref{E1} and \eqref{E2}, we eventually deduce that
\begin{equation*}
\# \mathcal{E}_{\delta}(X) \ll  X^{5/6 - \delta + \varepsilon} + X^{2/3+\varepsilon}.
\end{equation*}
Using the fact that $\delta < 1/6$ and choosing $\varepsilon = \delta/6$ we obtain the upper bound \eqref{Second'}, which completes the proof.
\end{proof}

\subsection{On the difference between the Weil height and the canonical height}

\label{Subsection heights}

We start by recalling the definitions of several height functions (see for instance \cite[Chapter VIII]{MR2514094} for more details). We let
$h : \mathbb{P}^1(\mathbb{Q}) \to \mathbb{R}_{\geq 0}$ be the logarithmic Weil height and we let
$h_x : \mathbb{P}^2(\mathbb{Q}) \to \mathbb{R}_{\geq 0}$ be defined by
\begin{equation*}
h_x(x:y:z) = h(x:z),
\end{equation*}
if $(x:y:z) \neq (0:1:0)$ and $h_x(0:1:0) = 0$. We also recall that the canonical height
$\hat{h}_E : E(\mathbb{Q}) \to \mathbb{R}_{\geq 0}$ on an elliptic curve $E$ is defined by setting
\begin{equation*}
\hat{h}_E(P) = \frac1{2} \lim_{n \to \infty} \frac1{4^n} h_x(2^n P),
\end{equation*}
for $P \in E(\mathbb{Q})$, and where $2^n P$ denotes the point obtained by adding the point $P$ to itself $2^n$ times using the group law of $E$.

In addition, we need to introduce the $j$-invariant of an elliptic curve. For $E_{A,B} \in \mathcal{F}$, we thus define
\begin{equation}
\label{j-invariant}
j(E_{A,B}) = - 1728 \frac{(4 A)^3}{\Delta(E_{A,B})}.
\end{equation}

The following result states that if $\delta \in (0,1)$ is fixed then for any $E_{A,B} \in \mathcal{F}_{\delta}(X)$, the height $h_x/2$ is essentially bounded above by the canonical height $\hat{h}_{E_{A,B}}$. We note that Alpoge \cite[Section $4$.$2$]{Alpoge} has recently made similar remarks.

\begin{lemma}
\label{Difference}
Let $\delta \in (0,1)$ be fixed. For any $E_{A,B} \in \mathcal{F}_{\delta}(X)$, we have the lower bound
\begin{equation*}
\hat{h}_{E_{A,B}} - \frac1{2} h_x \geq - \frac{\delta}{2} \log X - 3.
\end{equation*}
\end{lemma}

\begin{proof}
The inequality clearly holds at the point at infinity. Therefore, all along the proof, we let $P \in E_{A,B}(\mathbb{Q})$ be distinct from the point at infinity. We also let $x(P)$ denote the only rational number such that the coordinates of $P$ in $\mathbb{P}^2(\mathbb{Q})$ can be written as $(x(P):y(P):1)$ for some $y(P) \in \mathbb{Q}$.

Let $\lambda_{\infty}$ and $\lambda_p$, for any prime number $p$, denote the N\'{e}ron local heights relative to the point at infinity (see \cite[Chapter VI, Theorem $1$.$1$]{MR1312368}). Recall that the canonical height $\hat{h}_{E_{A,B}}$ can be decomposed into the sum of these local heights (see \cite[Chapter VI, Theorem $2$.$1$]{MR1312368}). We thus have
\begin{equation*}
\hat{h}_{E_{A,B}}(P) = \sum_p \lambda_p(P) + \lambda_{\infty}(P).
\end{equation*}

We now introduce some notation. For any prime number $p$ and any $n \in \mathbb{Z}_{\neq 0}$, we let $v_p(n)$ denote the $p$-adic valuation of $n$ and we define as usual
\begin{equation*}
|n|_p = \frac1{p^{v_p(n)}}.
\end{equation*}
In order to extend this definition to $\mathbb{Q}$, we let $|0|_p = 0$ and, for any $(r, s) \in \mathbb{Z} \times \mathbb{Z}_{\neq 0}$, we set
\begin{equation*}
\left| \frac{r}{s} \right|_p = \frac{|r|_p}{|s|_p}.
\end{equation*}
Finally, for $u \in \mathbb{R}$, it is convenient to set $\log^+ u = \log \max \{ 1, u \}$, so that for any $t \in \mathbb{Q}$, we have
\begin{equation*}
h(t:1) = \sum_p \log^+ | t |_p + \log^+ |t|.
\end{equation*}
This equality and \cite[Chapter VI, Theorem $4$.$1$]{MR1312368} give
\begin{equation*}
\hat{h}_{E_{A,B}}(P) - \frac1{2} h_x(P) = \sum_{p \mid \Delta(E_{A,B})} \left( \lambda_p(P) - \frac1{2} \log^+ | x(P) |_p  \right) +
\lambda_{\infty}(P) - \frac1{2} \log^+ | x(P)|.
\end{equation*}
Therefore, we are led to establishing lower bounds for the differences between the local heights. We start with the finite places and we distinguish two cases depending on whether $v_p(\Delta(E_{A,B})) \geq 2$ or $v_p(\Delta(E_{A,B})) = 1$.

For the prime numbers $p$ for which $v_p(\Delta(E_{A,B})) \geq 2$, we make use of the inequality
\cite[Theorem $4$.$1$, Part $(a)$]{MR1035944} which states that
\begin{equation*}
\lambda_p(P) - \frac1{2} \log^+ | x(P) |_p \geq - \frac1{24} \log^+ |j(E_{A,B})|_p.
\end{equation*}
Since $|j(E_{A,B})|_p \leq p^{v_p(\Delta(E_{A,B}))}$, we have $\log^+ |j(E_{A,B})|_p \leq \log p^{v_p(\Delta(E_{A,B}))}$. As a result, we get
\begin{equation}
\label{Inequality 0}
\sum_{v_p(\Delta(E_{A,B})) \geq 2} \! \left( \lambda_p(P) - \frac1{2} \log^+ | x(P) |_p \right) \geq
- \frac1{24} \sum_{v_p(\Delta(E_{A,B})) \geq 2} \! \log p^{v_p(\Delta(E_{A,B}))}.
\end{equation}
If $v_p(\Delta(E_{A,B}))$ is even then
\begin{equation*}
v_p(\Delta(E_{A,B})) = 2 v_p(\sq(\Delta(E_{A,B}))),
\end{equation*}
and if $v_p(\Delta(E_{A,B}))$ is odd then
\begin{equation*}
v_p(\Delta(E_{A,B})) = 2 v_p(\sq(\Delta(E_{A,B}))) + 1.
\end{equation*}
In both cases, if $v_p(\Delta(E_{A,B})) \geq 2$, we have
\begin{equation*}
v_p(\Delta(E_{A,B})) \leq 3 v_p(\sq(\Delta(E_{A,B}))).
\end{equation*}
Therefore, we get
\begin{equation*}
\prod_{v_p(\Delta(E_{A,B})) \geq 2} p^{v_p(\Delta(E_{A,B}))} \leq \sq(\Delta(E_{A,B}))^3.
\end{equation*}
Since $\sq(\Delta(E_{A,B})) \leq 4 X^{\delta}$, this eventually gives
\begin{equation}
\label{Product}
\prod_{v_p(\Delta(E_{A,B})) \geq 2} p^{v_p(\Delta(E_{A,B}))} \leq 4^3 X^{3 \delta}.
\end{equation}
The inequality \eqref{Inequality 0} thus implies
\begin{equation}
\label{Inequality 1}
\sum_{v_p(\Delta(E_{A,B})) \geq 2} \left( \lambda_p(P) - \frac1{2} \log^+ | x(P) |_p \right) \geq - \frac{\delta}{8} \log X - 0.2.
\end{equation}

For the prime numbers $p$ for which $v_p(\Delta(E_{A,B})) = 1$, we see that we necessarily have $v_p(6A) = 0$. Indeed, if we had $p=2$ or $p=3$ or $p \mid A$ then we would have $v_p(\Delta(E_{A,B})) \geq 2$. Thus, we also have $|j(E_{A,B})|_p = p$ by the definition \eqref{j-invariant} of $j(E_{A,B})$. As a result, we can use the inequality \cite[Theorem $4$.$1$, Part $(b)$]{MR1035944} which states that
\begin{equation*}
\lambda_p(P) - \frac1{2} \log^+ | x(P) |_p \geq \frac1{12} \log^+ |j(E_{A,B})|_p.
\end{equation*}
This immediately gives
\begin{equation*}
\sum_{v_p(\Delta(E_{A,B})) = 1} \left( \lambda_p(P) - \frac1{2} \log^+ | x(P) |_p \right) \geq
\frac1{12} \sum_{v_p(\Delta(E_{A,B})) = 1} \log p.
\end{equation*}
Using the inequality \eqref{Product}, we thus deduce that
\begin{equation}
\label{Inequality 2}
\sum_{v_p(\Delta(E_{A,B})) = 1} \left( \lambda_p(P) - \frac1{2} \log^+ | x(P) |_p \right) \geq
\frac1{12} \log | \Delta(E_{A,B}) | - \frac{\delta}{4} \log X - 0.4.
\end{equation}

Finally, we treat the case of the archimedean place. Using \cite[Theorem~$5$.$5$]{MR1035944}, we get
\begin{equation*}
\lambda_{\infty}(P) - \frac1{2} \log^+ | x(P) | \geq - \frac1{12} \log | \Delta(E_{A,B}) | - \frac1{8} \log^+ |j(E_{A,B})| - 1.
\end{equation*}
Since $|\Delta(E_{A,B})| > X^{1 - \delta}$ and $4 |A|^3 \leq X$ , we have $|j(E_{A,B})| \leq 1728 \cdot 4^2  X^{\delta}$. Therefore, we see that
\begin{equation}
\label{Inequality 3}
\lambda_{\infty}(P) - \frac1{2} \log^+ | x(P) | \geq - \frac1{12} \log | \Delta(E_{A,B}) | - \frac{\delta}{8} \log X - 2.3.
\end{equation}
Adding up the three inequalities \eqref{Inequality 1}, \eqref{Inequality 2} and \eqref{Inequality 3}, we finally deduce that
\begin{equation*}
\hat{h}_{E_{A,B}}(P) - \frac1{2} h_x(P) \geq - \frac{\delta}{2} \log X - 3,
\end{equation*}
which completes the proof.
\end{proof}

\subsection{Weierstrass equations and parametrization of rational points}

We describe here a classical procedure to parametrize rational points on elliptic curves defined by Weierstrass equations. The following result is, for instance, a particular case of \cite[Lemma~$1$]{MR3455753}.

\begin{lemma}
\label{Parametrization}
Let $(A,B) \in \mathbb{Z}^2$ be such that $4 A^3 + 27 B^2 \neq 0$. Let $(x,y,z) \in \mathbb{Z} \times \mathbb{Z}_{\geq 1}^2$ satisfying $\gcd(x,y,z) = 1$ and
\begin{equation*}
y^2 z = x^3 +A x z^2 + B z^3.
\end{equation*}
Then, there is a unique way to write $x = w q$ and $z = q^3$ where $(w, q) \in \mathbb{Z} \times \mathbb{Z}_{\geq 1}$ satisfies the condition $\gcd(w, q) = 1$ and the equation
\begin{equation*}
y^2 = w^3 + A w q^4 + B q^6.
\end{equation*}
\end{lemma}

\subsection{On the number of solutions to an inhomogeneous cubic congruence}

We now state a result which follows directly from the deep work of Baier and Browning \cite{MR3100953}. It provides a nice upper bound for the number of solutions to an inhomogeneous cubic congruence in two variables restricted to lie in a box, in the case where the modulus of the congruence is a fourth power.

For $U, V \geq 1$ and $q \in \mathbb{Z}_{\geq 1}$, we define
\begin{equation}
\label{Definition M}
\mathcal{M}(U,V,q) = \# \left\{ (u,v) \in \mathbb{Z}^2 :
\begin{array}{l}
v^2 = u^3 \imod{q^4} \\
|u| \leq U, |v| \leq V \\
\gcd(uv, q) = 1
\end{array}
\right\}.
\end{equation}
The following lemma is a particular case of the result of Baier and Browning \cite[Theorem $1$.$2$]{MR3100953}.

\begin{lemma}
\label{Baier-Browning}
Let $\varepsilon > 0$ be fixed. For $U, V \geq 1$ and $q \in \mathbb{Z}_{\geq 1}$, we have the upper bound
\begin{equation*}
\mathcal{M}(U,V,q) \ll (U V q)^{\varepsilon}
\left( \frac{UV}{q^4} + \frac{V}{q^2} + \frac{U^{3/4} V^{1/2}}{q^2} + \frac{q^2 U}{V} + \frac{q^2}{U^{1/2}} \right).
\end{equation*}
Therefore, if $q \ll U^{1/2}$ and $U \ll V^{2/3}$ then we have the upper bound
\begin{equation*}
\mathcal{M}(U,V,q) \ll V^{\varepsilon} \frac{UV}{q^4}.
\end{equation*}
\end{lemma}

It is clear that $\mathcal{M}(U,V,q)$ is expected to be of size $UV/q^4$ provided that $U$ and $V$ are large enough in terms of $q$. Therefore, we see that the upper bound available under the assumptions $q \ll U^{1/2}$ and $U \ll V^{2/3}$ is as good as possible up to the factor $V^{\varepsilon}$.

It is also worth pointing out that Lemma \ref{Baier-Browning} is stronger than the upper bound which would follow from a direct application of Weil's bound (consequence of the Riemann Hypothesis for curves defined over finite fields) for the exponential sums involved in this counting problem.

\subsection{On the number of elliptic curves with a non-torsion rational point of small height}

For $X, T \geq 1$, we define
\begin{equation}
\label{Definition N}
\mathcal{N}(X,T) = \# \{ E_{A,B} \in \mathcal{F}(X) : \mu(E_{A,B}) \leq \log T \}.
\end{equation}
In the following lemma, we establish an upper bound for the quantity $\mathcal{N}(X,T)$ which shall be considered as the key tool in the proof of Theorem \ref{Main Theorem}. We remark that this upper bound is nontrivial provided that $T \leq X^{7/24-\eta}$ for some
$\eta > 0$. Moreover, we should definitely stress that any improvement on this exponent would directly translate into an improvement on the constant $7/24$ appearing in the statement of Theorem \ref{Main Theorem}.

\begin{lemma}
\label{Key Lemma}
Let $\delta \in (0,1/6)$ be fixed. For $X, T \geq 1$, we have the upper bound
\begin{equation*}
\mathcal{N}(X,T) \ll (XT)^{2 \delta} \left( X^{3/4} + X^{5/12} T + X^{1/4} T^2 +  X^{-1/12} T^3 \right) + X^{5/6 - 5\delta/6}.
\end{equation*}
\end{lemma}

\begin{proof}
Recall the definition \eqref{Definition F_delta} of $\mathcal{F}_{\delta}(X)$. Our first aim is to apply Lemma~\ref{Difference}. In order to do so, we let $\delta \in (0,1/6)$ and we note that Lemma~\ref{Size F delta} implies in particular that
\begin{equation}
\label{Intermediate 1}
\mathcal{N}(X,T) = \mathcal{N}_{\delta}(X,T) + O (X^{5/6 - 5\delta/6}),
\end{equation}
where we have set
\begin{equation*}
\mathcal{N}_{\delta}(X,T) = \# \{ E_{A,B} \in \mathcal{F}_{\delta}(X) : \mu(E_{A,B}) \leq \log T \}.
\end{equation*}
In addition, we clearly have
\begin{equation*}
\mathcal{N}_{\delta}(X,T) \leq \sum_{E_{A,B} \in \mathcal{F}_{\delta}(X)}
\# \{ P \in E_{A,B}(\mathbb{Q}) \smallsetminus E_{A,B}(\mathbb{Q})_{\tors} : \hat{h}_{E_{A,B}}(P) \leq \log T \}.
\end{equation*}
Using Lemma \ref{Difference}, we obtain
\begin{equation*}
\mathcal{N}_{\delta}(X,T) \ll \sum_{E_{A,B} \in \mathcal{F}_{\delta}(X)}
\# \{ P \in E_{A,B}(\mathbb{Q}) \smallsetminus E_{A,B}(\mathbb{Q})_{\tors} : \exp h_x(P) \ll T^2 X^{\delta} \}.
\end{equation*}
Now that we have used Lemma \ref{Difference}, we no longer need to restrict our summation to the set $\mathcal{F}_{\delta}(X)$. We thus write
\begin{equation*}
\mathcal{N}_{\delta}(X,T) \ll \sum_{E_{A,B} \in \mathcal{F}(X)}
\# \{ P \in E_{A,B}(\mathbb{Q}) \smallsetminus E_{A,B}(\mathbb{Q})_{\tors} : \exp h_x(P) \ll T^2 X^{\delta} \}.
\end{equation*}
It follows from Lemma \ref{Parametrization} that
\begin{equation*}
\mathcal{N}_{\delta}(X,T) \ll \# \mathcal{S}(X,T),
\end{equation*}
where we have set
\begin{equation*}
\mathcal{S}(X,T) = 
\left\{ (A, B,w,y,q) \in \mathbb{Z}^3 \times \mathbb{Z}_{\geq 1}^2 :
\begin{array}{l}
y^2 = w^3 + A w q^4 + B q^6 \\
4 |A|^3, 27 B^2 \leq X \\
w, q^2 \ll T^2 X^{\delta} \\
\gcd(w, q) = 1
\end{array}
\right\}.
\end{equation*}
We define
\begin{equation*}
\mathcal{S}^{(1)}(X,T) = \{ (A, B,w,y,q) \in \mathcal{S}(X,T) : |w| > X^{1/6} q^2 \},
\end{equation*}
and $\mathcal{S}^{(2)}(X,T) = \mathcal{S}(X,T) \smallsetminus \mathcal{S}^{(1)}(X,T)$, so that
\begin{equation}
\label{S1 S2}
\mathcal{N}_{\delta}(X,T) \ll \# \mathcal{S}^{(1)}(X,T) + \# \mathcal{S}^{(2)}(X,T).
\end{equation}

We start by proving an upper bound for the cardinality of the set $\mathcal{S}^{(1)}(X,T)$. We note that the equation
$y^2 = w^3 + A w q^4 + B q^6$ and the condition $|w| > X^{1/6} q^2$ imply that $w > 0$ and
\begin{equation*}
y^2 = w^3 \left( 1 + O \left( \frac{X^{1/3} q^4}{w^2} \right) \right).
\end{equation*}
Therefore, we also have
\begin{equation*}
y = w^{3/2} + O \left( \frac{X^{1/3} q^4}{w^{1/2}} \right).
\end{equation*}
We thus get
\begin{equation*}
\# \mathcal{S}^{(1)}(X,T) \ll
\sum_{\substack{w, q^2 \ll T^2 X^{\delta} \\ w > X^{1/6} q^2 \\ \gcd(w, q) = 1}}
\# \left\{ (A,B,y) \in \mathbb{Z}^2 \times \mathbb{Z}_{\geq 1} :
\begin{array}{l}
y^2 = w^3 + A w q^4 + B q^6 \\
|A|^3, B^2 \leq X \\
y - w^{3/2} \ll X^{1/3} q^4 w^{-1/2}
\end{array}
\right\}.
\end{equation*}
For fixed $(w,q) \in \mathbb{Z}_{\geq 1}^2$ satisfying $\gcd(w, q) = 1$ and for fixed $y \in \mathbb{Z}_{\geq 1}$ and
$B \in \mathbb{Z}$, we see that the equation $y^2 = w^3 + A w q^4 + B q^6$ has a solution $A \in \mathbb{Z}$ if and only if the two congruences $y^2 = B q^6 \imod{w}$ and $y^2 = w^3 \imod{q^4}$ hold. It follows from this observation that
\begin{equation*}
\# \mathcal{S}^{(1)}(X,T) \ll
\sum_{\substack{w, q^2 \ll T^2 X^{\delta} \\ w > X^{1/6} q^2 \\ \gcd(w, q) = 1}}
\# \left\{ (B,y) \in \mathbb{Z} \times \mathbb{Z}_{\geq 1} :
\begin{array}{l}
y^2 = B q^6 \imod{w} \\
y^2 = w^3 \imod{q^4} \\
|B| \leq X^{1/2} \\
y - w^{3/2} \ll X^{1/3} q^4 w^{-1/2}
\end{array}
\right\}.
\end{equation*}
For fixed $(w,q) \in \mathbb{Z}_{\geq 1}^2$ satisfying $\gcd(w, q) = 1$ and fixed $y \in \mathbb{Z}_{\geq 1}$, we use the trivial upper bound 
\begin{equation*}
\# \left\{ B \in \mathbb{Z} :
\begin{array}{l}
|B| \leq X^{1/2} \\
y^2 = B q^6 \imod{w}
\end{array}
\right\} \ll \frac{X^{1/2}}{w} + 1.
\end{equation*}
We deduce
\begin{equation*}
\# \mathcal{S}^{(1)}(X,T) \ll
\sum_{\substack{w, q^2 \ll T^2 X^{\delta} \\ w > X^{1/6} q^2 \\ \gcd(w, q) = 1}}
\# \left\{ y \in \mathbb{Z}_{\geq 1} :
\begin{array}{l}
y^2 = w^3 \imod{q^4} \\
y - w^{3/2} \ll X^{1/3} q^4 w^{-1/2}
\end{array}
\right\}
\left( \frac{X^{1/2}}{w} + 1\right).
\end{equation*}
For fixed $(w,q) \in \mathbb{Z}_{\geq 1}^2$ satisfying $\gcd(w, q) = 1$, we make use of the upper bound
\begin{equation*}
\# \left\{ y \in \mathbb{Z}_{\geq 1} :
\begin{array}{l}
y^2 = w^3 \imod{q^4} \\
y - w^{3/2} \ll X^{1/3} q^4 w^{-1/2}
\end{array}
\right\} \ll q^{\varepsilon} \left( \frac{X^{1/3}}{w^{1/2}} + 1 \right),
\end{equation*}
for any fixed $\varepsilon > 0$. This gives
\begin{equation*}
\# \mathcal{S}^{(1)}(X,T) \ll (X T)^{\varepsilon}
\sum_{\substack{w, q^2 \ll T^2 X^{\delta} \\ w > X^{1/6} q^2}}
\left( \frac{X^{1/3}}{w^{1/2}} + 1 \right) \left( \frac{X^{1/2}}{w} + 1\right).
\end{equation*}
We finally obtain
\begin{equation}
\label{S1}
\# \mathcal{S}^{(1)}(X,T) \ll (X T)^{\varepsilon}
\left( X^{3/4} + X^{5/12 + \delta/2} T + X^{1/4 + \delta} T^2 +  X^{-1/12 + 3 \delta/2} T^3 \right).
\end{equation}

We now prove an upper bound for the cardinality of the set $\mathcal{S}^{(2)}(X,T)$. We note that the equation
$y^2 = w^3 + A w q^4 + B q^6$ and the condition $|w| \leq X^{1/6} q^2$ imply that $y \leq 2 X^{1/4} q^3$ and
\begin{equation}
\label{B}
B = \frac{y^2 - w^3}{q^6} + O \left( \frac{X^{1/3} w}{q^2} \right).
\end{equation}
Therefore, we have
\begin{equation*}
\# \mathcal{S}^{(2)}(X,T) \ll \sum_{q \ll T X^{\delta/2}}
\# \left\{ (A,B,w,y) \in \mathbb{Z}^3 \times \mathbb{Z}_{\geq 1} :
\begin{array}{l}
y^2 = w^3 + A w q^4 + B q^6 \\
\eqref{B} \\
w \ll \min \{ T^2 X^{\delta}, X^{1/6} q^2 \} \\
y \leq 2 X^{1/4} q^3 \\
\gcd(w, q) = 1
\end{array}
\right\}.
\end{equation*}
Recall the definition \eqref{Definition M} of $\mathcal{M}(U,V,q)$. Reasoning as in the first case to count the number of admissible
$(A,B) \in \mathbb{Z}^2$ for fixed $(w,q) \in \mathbb{Z} \times \mathbb{Z}_{\geq 1}$ satisfying $\gcd(w, q) = 1$ and fixed
$y \in \mathbb{Z}_{\geq 1}$, we obtain
\begin{equation*}
\# \mathcal{S}^{(2)}(X,T) \ll \sum_{q \ll T X^{\delta/2}}
\mathcal{M}(\min \{ T^2 X^{\delta}, X^{1/6} q^2 \}, 2 X^{1/4} q^3, q) \left( \frac{X^{1/3}}{q^2} + 1\right).
\end{equation*}
Since we clearly have the upper bounds 
\begin{equation*}
q \ll \min \{ T^2 X^{\delta}, X^{1/6} q^2 \}^{1/2},
\end{equation*}
and
\begin{equation*}
\min \{T^2 X^{\delta}, X^{1/6} q^2 \} \ll (X^{1/4} q^3)^{2/3},
\end{equation*}
Lemma \ref{Baier-Browning} gives
\begin{equation*}
\mathcal{M}(\min \{ T^2 X^{\delta}, X^{1/6} q^2 \}, 2 X^{1/4} q^3, q) \ll (XT)^{\varepsilon}
\frac{\min \{T^2 X^{\delta}, X^{1/6} q^2 \} X^{1/4}}{q}.
\end{equation*}
Therefore, we deduce that
\begin{equation*}
\# \mathcal{S}^{(2)}(X,T) \ll (XT)^{\varepsilon}
\sum_{q \ll T X^{\delta/2}} \left( \frac{X^{3/4}}{q} + \frac{X^{1/4 + \delta} T^2}{q} \right),
\end{equation*}
which eventually gives
\begin{equation}
\label{S2}
\# \mathcal{S}^{(2)}(X,T) \ll (XT)^{\varepsilon} \left(X^{3/4} + X^{1/4 + \delta} T^2 \right).
\end{equation}
Recalling the upper bounds \eqref{S1 S2}, \eqref{S1} and \eqref{S2}, we see that we have
\begin{equation*}
\mathcal{N}_{\delta}(X,T) \ll (XT)^{\varepsilon}
\left( X^{3/4} + X^{5/12 + \delta/2} T + X^{1/4 + \delta} T^2 +  X^{-1/12 + 3 \delta/2} T^3 \right).
\end{equation*}
Recalling the estimate \eqref{Intermediate 1}, we eventually obtain
\begin{equation*}
\mathcal{N}(X,T) \ll (XT)^{\varepsilon}
\left( X^{3/4} + X^{5/12 + \delta/2} T + X^{1/4 + \delta} T^2 +  X^{-1/12 + 3 \delta/2} T^3 \right) + X^{5/6 - 5\delta/6}.
\end{equation*}
Choosing $\varepsilon = \delta/2$ completes the proof.
\end{proof}

\subsection{On the number of elliptic curves with positive rank}

We now give a lower bound for the cardinality of the set $\mathcal{F}_{\geq 1}(X)$. The result of Bhargava and Skinner
\cite[Theorem $1$]{MR3237733} (later refined by the same authors together with Zhang \cite[Theorem $3$]{BhargavaSkinnerZhang}) states that
\begin{equation*}
\# \{ E_{A,B} \in \mathcal{F}(X) : \rank E_{A,B}(\mathbb{Q}) = 1 \} \gg \# \mathcal{F}(X),
\end{equation*}
so in particular we have the following lemma.

\begin{lemma}
\label{Bhargava-Skinner}
We have the lower bound 
\begin{equation*}
\# \mathcal{F}_{\geq 1}(X) \gg \# \mathcal{F}(X).
\end{equation*}
\end{lemma}

\section{Proof of Theorem \ref{Main Theorem}}

Our aim is to prove that for any fixed $\vartheta \in (0,7/24)$, we have the equality
\begin{equation*}
\lim_{X \to \infty}
\frac{\# \{ E_{A,B} \in \mathcal{F}_{\geq 1}(X) : \mu(E_{A,B}) \leq (7/24 - \vartheta) \log H(E_{A,B}) \}}{\# \mathcal{F}_{\geq 1}(X)} = 0.
\end{equation*}
Recalling the definition \eqref{Definition N} of $\mathcal{N}(X,T)$, we see that
\begin{equation*}
\# \{ E_{A,B} \in \mathcal{F}_{\geq 1}(X) : \mu(E_{A,B}) \leq (7/24 - \vartheta) \log H(E_{A,B}) \} \leq
\mathcal{N}(X, X^{7/24 - \vartheta}).
\end{equation*}
We can clearly assume without loss of generality that $\vartheta < 1/24$. Therefore, for any fixed $\delta \in (0,1/6)$,
Lemma \ref{Key Lemma} implies that
\begin{equation*}
\# \{ E_{A,B} \in \mathcal{F}_{\geq 1}(X) : \mu(E_{A,B}) \leq (7/24 - \vartheta) \log H(E_{A,B}) \} \ll
X^{5/6 - 2 \vartheta + 3 \delta} + X^{5/6 - 5 \delta /6}.
\end{equation*}
Choosing for instance $\delta = 2 \vartheta/5$, we deduce that
\begin{equation*}
\# \{ E_{A,B} \in \mathcal{F}_{\geq 1}(X) : \mu(E_{A,B}) \leq (7/24 - \vartheta) \log H(E_{A,B}) \} \ll X^{5/6 - \vartheta/3}.
\end{equation*}
Using Lemma \ref{Bhargava-Skinner} and the estimate \eqref{Size F}, we eventually obtain
\begin{equation*}
\frac{\# \{ E_{A,B} \in \mathcal{F}_{\geq 1}(X) : \mu(E_{A,B}) \leq (7/24 - \vartheta) \log H(E_{A,B}) \}}{\# \mathcal{F}_{\geq 1}(X)} \ll X^{-\vartheta/3},
\end{equation*}
which completes the proof of Theorem \ref{Main Theorem}.

\bibliographystyle{amsalpha}
\bibliography{biblio}

\newcommand{\etalchar}[1]{$^{#1}$}
\providecommand{\bysame}{\leavevmode\hbox to3em{\hrulefill}\thinspace}
\providecommand{\MR}{\relax\ifhmode\unskip\space\fi MR }
\providecommand{\MRhref}[2]{%
  \href{http://www.ams.org/mathscinet-getitem?mr=#1}{#2}
}
\providecommand{\href}[2]{#2}
\begin{thebibliography}{BKL{\etalchar{+}}15}

\bibitem[Alp15]{Alpoge}
L.~Alpoge, \emph{The average number of integral points on elliptic curves is
  bounded}, arXiv:1412.1047v3 (2015).

\bibitem[BB13]{MR3100953}
S.~Baier and T.~D. Browning, \emph{Inhomogeneous cubic congruences and rational
  points on del {P}ezzo surfaces}, J. Reine Angew. Math. \textbf{680} (2013),
  69--151.

\bibitem[BKL{\etalchar{+}}15]{MR3393023}
M.~Bhargava, D.~M. Kane, H.~W. Lenstra, Jr., B.~Poonen, and E.~Rains,
  \emph{Modeling the distribution of ranks, {S}elmer groups, and
  {S}hafarevich-{T}ate groups of elliptic curves}, Camb. J. Math. \textbf{3}
  (2015), no.~3, 275--321.

\bibitem[Bru92]{MR1176198}
A.~Brumer, \emph{The average rank of elliptic curves. {I}}, Invent. Math.
  \textbf{109} (1992), no.~3, 445--472.

\bibitem[BS13a]{BhargavaShankar1}
M.~Bhargava and A.~Shankar, \emph{The average number of elements in the
  $4$-{S}elmer groups of elliptic curves is $7$}, arXiv:1312.7333v1 (2013).

\bibitem[BS13b]{BhargavaShankar2}
\bysame, \emph{The average size of the $5$-{S}elmer group of elliptic curves is
  $6$, and the average rank is less than $1$}, arXiv:1312.7859v1 (2013).

\bibitem[BS14]{MR3237733}
M.~Bhargava and C.~Skinner, \emph{A positive proportion of elliptic curves over
  {$\mathbb{Q}$} have rank one}, J. Ramanujan Math. Soc. \textbf{29} (2014),
  no.~2, 221--242.

\bibitem[BS15a]{MR3272925}
M.~Bhargava and A.~Shankar, \emph{Binary quartic forms having bounded
  invariants, and the boundedness of the average rank of elliptic curves}, Ann.
  of Math. (2) \textbf{181} (2015), no.~1, 191--242.

\bibitem[BS15b]{MR3275847}
\bysame, \emph{Ternary cubic forms having bounded invariants, and the existence
  of a positive proportion of elliptic curves having rank 0}, Ann. of Math. (2)
  \textbf{181} (2015), no.~2, 587--621.

\bibitem[BSZ14]{BhargavaSkinnerZhang}
M.~Bhargava, C.~Skinner, and W.~Zhang, \emph{A majority of elliptic curves over
  $\mathbb{Q}$ satisfy the {B}irch and {S}winnerton-{D}yer conjecture},
  arXiv:1407.1826v2 (2014).

\bibitem[CKRS02]{MR1956231}
J.~B. Conrey, J.~P. Keating, M.~O. Rubinstein, and N.~C. Snaith, \emph{On the
  frequency of vanishing of quadratic twists of modular {$L$}-functions},
  Number theory for the millennium, {I} ({U}rbana, {IL}, 2000), A K Peters,
  Natick, MA, 2002, pp.~301--315.

\bibitem[Del01]{MR1837670}
C.~Delaunay, \emph{Heuristics on {T}ate-{S}hafarevitch groups of elliptic
  curves defined over {$\mathbb{Q}$}}, Experiment. Math. \textbf{10} (2001),
  no.~2, 191--196.

\bibitem[Est31]{MR1512732}
T.~Estermann, \emph{Einige {S}\"atze \"uber quadratfreie {Z}ahlen}, Math. Ann.
  \textbf{105} (1931), no.~1, 653--662.

\bibitem[Gol79]{MR564926}
D.~Goldfeld, \emph{Conjectures on elliptic curves over quadratic fields},
  Number theory, {C}arbondale 1979 ({P}roc. {S}outhern {I}llinois {C}onf.,
  {S}outhern {I}llinois {U}niv., {C}arbondale, {I}ll., 1979), Lecture Notes in
  Math., vol. 751, Springer, Berlin, 1979, pp.~108--118.

\bibitem[HS88]{MR948108}
M.~Hindry and J.~H. Silverman, \emph{The canonical height and integral points
  on elliptic curves}, Invent. Math. \textbf{93} (1988), no.~2, 419--450.

\bibitem[Iwa90]{MR1081731}
H.~Iwaniec, \emph{On the order of vanishing of modular {$L$}-functions at the
  critical point}, S\'em. Th\'eor. Nombres Bordeaux (2) \textbf{2} (1990),
  no.~2, 365--376.

\bibitem[KS99]{MR1659828}
N.~M. Katz and P.~Sarnak, \emph{Random matrices, {F}robenius eigenvalues, and
  monodromy}, American Mathematical Society Colloquium Publications, vol.~45,
  American Mathematical Society, Providence, RI, 1999.

\bibitem[Lan78]{MR518817}
S.~Lang, \emph{Elliptic curves: {D}iophantine analysis}, Grundlehren der
  Mathematischen Wissenschaften [Fundamental Principles of Mathematical
  Sciences], vol. 231, Springer-Verlag, Berlin-New York, 1978.

\bibitem[Lan83]{MR717593}
\bysame, \emph{Conjectured {D}iophantine estimates on elliptic curves},
  Arithmetic and geometry, {V}ol. {I}, Progr. Math., vol.~35, Birkh\"auser
  Boston, Boston, MA, 1983, pp.~155--171.

\bibitem[LB16]{MR3455753}
P.~Le~Boudec, \emph{Height of rational points on quadratic twists of a given
  elliptic curve}, Bull. Lond. Math. Soc. \textbf{48} (2016), no.~1, 99--108.

\bibitem[MM91]{MR1109350}
M.~R. Murty and V.~K. Murty, \emph{Mean values of derivatives of modular
  {$L$}-series}, Ann. of Math. (2) \textbf{133} (1991), no.~3, 447--475.

\bibitem[Sil81]{MR630588}
J.~H. Silverman, \emph{Lower bound for the canonical height on elliptic
  curves}, Duke Math. J. \textbf{48} (1981), no.~3, 633--648.

\bibitem[Sil84]{MR747871}
\bysame, \emph{Lower bounds for height functions}, Duke Math. J. \textbf{51}
  (1984), no.~2, 395--403.

\bibitem[Sil90]{MR1035944}
\bysame, \emph{The difference between the {W}eil height and the canonical
  height on elliptic curves}, Math. Comp. \textbf{55} (1990), no.~192,
  723--743.

\bibitem[Sil94]{MR1312368}
\bysame, \emph{Advanced topics in the arithmetic of elliptic curves}, Graduate
  Texts in Mathematics, vol. 151, Springer-Verlag, New York, 1994.

\bibitem[Sil09]{MR2514094}
\bysame, \emph{The arithmetic of elliptic curves}, second ed., Graduate Texts
  in Mathematics, vol. 106, Springer, Dordrecht, 2009.

\bibitem[Sil10]{MR2645981}
\bysame, \emph{Lang's height conjecture and {S}zpiro's conjecture}, New York J.
  Math. \textbf{16} (2010), 1--12.

\bibitem[Szp90]{MR1065151}
L.~Szpiro, \emph{Discriminant et conducteur des courbes elliptiques},
  Ast\'erisque (1990), no.~183, 7--18, S\'eminaire sur les Pinceaux de Courbes
  Elliptiques (Paris, 1988).

\end{thebibliography}

\end{document}